\pdfoutput=1
\documentclass[11pt]{amsart}

\usepackage[utf8]{inputenc}
\usepackage[T1]{fontenc}
\usepackage[british]{babel}
\usepackage{stmaryrd}
\usepackage{url}
\usepackage{color}
\usepackage{graphicx}
\usepackage{mathrsfs}

\usepackage{mathtools}
\numberwithin{equation}{section}

\usepackage{microtype}
\mathtoolsset{centercolon}
\usepackage{enumitem}
\usepackage{amssymb}
\usepackage{bm}

\usepackage[dvipsnames]{xcolor}
\usepackage{hyperref}
\hypersetup{
	colorlinks,
	linkcolor={red!50!black},
	citecolor={blue!50!black},
	urlcolor={blue!80!black}
}
\usepackage{amsthm}
\usepackage[margin=0pt,small]{caption} 
\usepackage{geometry}
\usepackage{cleveref}

\setlist{itemsep=1pt,topsep=1pt,parsep=1pt}

\newcommand*{\EE}{\mathbb{E}}
\newcommand*{\PP}{\mathbb{P}}
\newcommand*{\bE}{\mathbf{E}}
\newcommand*{\bP}{\mathbf{P}}
\newcommand*{\ind}{\mathbf{1}}
\newcommand*{\numberset}{\mathbb}
\newcommand*{\N}{\numberset{N}}
\newcommand*{\Z}{\numberset{Z}}
\newcommand*{\R}{\numberset{R}}

\newcommand{\e}{\mathrm{e}}
\newcommand*{\tF}{\mathtt{F}}
\newcommand*{\tf}{\mathfrak{f}}
\newcommand*{\cons}{\mathrm{cons}}

\newcommand*{\Var}{\mathbb{V}\mathrm{ar}}
\newcommand*{\Cov}{\mathbb{C}\mathrm{ov}}
\newcommand*\dd{\mathop{}\!\mathrm{d}}

\newcommand{\gep}{\varepsilon}

\theoremstyle{plain}
\newtheorem{theorem}{Theorem}[section]

\newtheorem{lemma}[theorem]{Lemma}

\newtheorem{proposition}[theorem]{Proposition}

\newtheorem{theoremalpha}{Theorem}

\theoremstyle{definition}

\theoremstyle{remark}
\newtheorem{remark}[theorem]{Remark}

\title[Free energy of \(2d\) directed polymers]{Sharp behavior of the free energy for the two-dimensional directed polymer model}
\author{Quentin Berger}
\address{Université Sorbonne Paris Nord, Laboratoire d'Analyse, Géométrie et Applications, CNRS UMR 7539, 99 Av. J-B Clément, 93430 Villetaneuse, France and Institut Universitaire de France}
\email{quentin.berger@math.univ-paris13.fr}
\author{Shuta Nakajima} 
\address{Department of Mathematics, Keio University}
\email{njima@keio.jp}


\begin{document}

\begin{abstract}
  \noindent
  We consider the directed polymer model on \(\mathbb{Z}^d\), focusing on the critical dimension \(d=2\).
  Our main contribution is to give a sharp lower bound on the free energy in the high-temperature regime.
  Our proof uses a percolation argument inspired by \cite{Lac10a}, for which we introduce a key property of bounded ``\(\log\)-energy'': this property quantifies the regularity of the polymer measures at diffusive scales and we show that it propagates along open paths.
  Combined with Theorem~2.8 of~\cite{BCT25}, this gives upper and lower bounds on the free energy that are matching up to a multiplicative constant.
\end{abstract}

\maketitle

\section{Introduction and main results}

The directed polymer model in random environment is a prototypical disordered system, introduced by Huse and Henley~\cite{HH85} in 1985 in dimension \(1+1\), as an effective interface model for the Ising model with random impurities. 
It was then generalized to dimensions \(1+d\), \(d\geq 1\), in order to describe a stretched polymer in a heterogeneous solvent. 
The model, despite its relative simplicity, exhibits a very rich behavior and has attracted a lot of attention over the last 40 years, in particular due to its close connection to the stochastic heat equation and the KPZ equation.
We refer to \cite{Com17} for a comprehensive overview and to \cite{Zyg24} for an account of recent developments.

The present article studies the free energy of the model, which is a central physical quantity encoding its phase transition and is related to some localization properties. 
We focus here on the critical dimension \(d=2\), and we improve existing bounds (in chronological order: \cite{Lac10a,Nak14,BL17,BCT25}) to obtain the sharp behavior of the free energy in the high-temperature regime.

\subsection{Introduction of the model}

Let \((S_n)_{n\geq 0}\) be a simple random walk on \(\Z^d\), \(d\geq 1\), whose law we denote by \(\bP_x\) when the starting point is \(S_0=x\).
For simplicity, we also denote \(\bP=\bP_0\).
Let also \(\omega = (\omega_{n,x})_{n\geq 0,x\in \Z^d}\) be a field of i.i.d.\ random variables, whose law is denoted by~\(\PP\). We assume that the~\(\omega_{n,x}\) are centered and normalized and that they have a finite exponential moment:
\begin{equation}
       \label{hyp:omega}
       \EE[\omega_{1,0}] =0, \quad \Var(\omega_{1,0})=1\,,
       \quad \text{ and } \quad \lambda (\beta) := \log \EE[\e^{\beta\omega_{1,0}}] <+\infty \ \text{for } \beta \in (-\beta_0,\beta_0) \,,
\end{equation}
for some \(\beta_0>0\).

The random walk trajectory \((n,S_n)_{n\geq 0}\) represents the trajectory of a directed polymer, and the random environment \(\omega\) represents random impurities in the solvent.
For a fixed realization of the environment \(\omega\) (quenched disorder) and for an inverse temperature \(\beta\geq 0\), we define the following Gibbs measures, referred to as polymer measures: for \(N\in \N\),
\begin{equation*}
       \frac{\dd \bP_{N}^{\beta,\omega}}{\dd \bP}(S) = \frac{1}{\hat{Z}_{N}^{\beta,\omega}} \exp\Big( \beta\sum_{k=1}^{N}  \omega_{k,S_k}\Big) = \frac{1}{Z_{N}^{\beta,\omega}} \exp\Big( \sum_{k=1}^{N} \big( \beta \omega_{k,S_k} -\lambda(\beta) \big)\Big) \,,
\end{equation*}
where the last identity holds for \(\beta<\beta_0\).
Note that the two formulations are equivalent since the factor \(\e^{-\lambda(\beta)N}\) is absorbed in the partition functions which normalize~\(\bP_{N}^{\beta,\omega}\) to a probability measure: \(Z_{N}^{\beta,\omega} =\e^{-\lambda(\beta)N}\hat{Z}_{N}^{\beta,\omega}\). 
The only effect of this normalization is to restrict \(\beta\) to \([0,\beta_0)\), but this is not an issue since we only consider the high-temperature behavior.
In the following, we focus on the partition function \(Z_{N}^{\beta,\omega}\), which can be written as
\[
Z_N^{\beta,\omega} =\bE\Big[ \e^{\sum_{k=1}^{N}(\beta\omega_{k,S_k} - \lambda(\beta))} \Big] \,.
\]

\subsection{Free energy and phase transition}

An important physical quantity is the \emph{free energy} of the model, defined by the following limit, which exists \(\PP\)-a.s.\ and in \(L^q(\PP)\) for all \(q\geq 1\):
\begin{equation}
       \label{def:free-energy}
       \tf(\beta) = \lim_{N\to\infty} \frac1N \log Z_N^{\beta,\omega} = \lim_{N\to\infty} \frac1N \EE \big[\log Z_N^{\beta,\omega}\big] \,.
\end{equation}
We refer to \cite[Theorem~2.1]{Com17} for a proof.
(In practice, \cite{Com17} rather considers the \(\PP\)-a.s.\ limit \(p(\beta) = \lim_{N\to\infty} \frac1N \log \hat{Z}_{N}^{\beta,\omega}\), but we simply have \(\tf(\beta) = p(\beta)-\lambda(\beta)\).)

One can show that the free energy \(\tf(\beta)\) is a continuous non-increasing function of \(\beta\), see \cite[Theorem~2.3]{Com17}, so in particular, there exists a critical point
\[
\beta_c := \sup\{\beta,\; \tf(\beta) =0\}  \in [0,\infty] \,, \quad \text{ so that }\quad \tf(\beta)<0 \ \Leftrightarrow \ \beta>\beta_c \,.
\]
One can show that \(\beta_c<+\infty\) in generic cases, namely if \(\PP(\omega_{1,0}= \mathrm{ess}\sup \omega_{1,0}) <\frac{1}{2d}\), see \cite[Corollary~2.1]{Com17}.
Another natural question, related to the question of disorder relevance, is whether one has \(\beta_c=0\).
It has been shown that \(\beta_c=0\) in dimension \(d=1,2\), see~\cite{CV06,Lac10a} and that \(\beta_c >0\) in dimension \(d\geq 3\), see~\cite{Bol89}.

As far as the behavior of the free energy near the critical point~\(\beta_c\) is concerned, precise results have been obtained in dimension \(d=1\) in~\cite{AY15,Lac10a,Nak19} and in dimension \(d=2\) in~\cite{BCT25,BL17,Lac10a,Nak14}.
The case of dimension~\(d\geq 3\) has remained open until the recent breakthroughs~\cite{JL26,Lac25}.
Let us summarize the results in the following statement\footnote{In the case of dimension \(d=2\), we state~\cite{BL17} instead of the more recent~\cite{BCT25} to avoid introducing some notation.}.

\begin{theoremalpha}
       \label{thm:existing}
       We have the following critical behaviors for the free energy:
       \begin{enumerate}[label=(\roman*)]
              \item 
              In dimension \(d=1\), \cite{Nak19} shows
              \[
              \tf(\beta) = - (1+o(1)) \frac16 \, \beta^4 \quad \text{ as } \beta\downarrow 0 \,.
              \]

              \item \label{item-ii}
              In dimension \(d=2\), \cite{BL17} shows 
              \[
              \tf(\beta) =  -\, \exp\Big( - (1+o(1)) \frac{\pi}{\beta^2} \Big) \quad \text{ as } \beta\downarrow 0 \,.
              \]

              \item 
              In dimension \(d\geq 3\), \cite{JL26} shows (the upper bound holds for Gaussian disorder) 
              \[
               -\, \exp\Big(- u^{- \frac{2+d}{10 d} +o(1)} \Big) \leq \tf(\beta_c +u) \leq -\, \exp\Big(- u^{-1+o(1)} \Big) \quad \text{ as } u\downarrow 0 \,.
              \]
       \end{enumerate}
\end{theoremalpha}

The free energy is also intimately related to localization properties of the polymer measure: in particular,  \(\beta_c\) marks a localization transition for the directed polymer.
Indeed, one has that trajectories are diffusive at weak disorder and satisfy an invariance principle, see \cite{CY06} or \cite{Lac25-ecp} for a short proof (note that weak disorder holds for any \(\beta \in [0,\beta_c]\) thanks to~\cite{JL24a}). 
Additionally, endpoint localization has been proven when \(\beta>\beta_c\), see \cite{CSY03} or \cite{BC20} for the most advanced results.

In fact, the free energy gives some quantitative estimates on localization properties, in terms of the end-point or replica overlap, see \cite{CH02,CH06,CSY03}. 
To state two results, \cite{CH06} and \cite{CH02} respectively show that\footnote{The article~\cite{CH06} actually considers a continuous-time version of the model.}, for a Gaussian environment \((\omega_{1,0}\sim \mathcal{N}(0,1))\),  $\PP\text{-a.s.}$
\begin{align*}
- \tf(\beta)
&=
\lim_{N\to\infty}\frac1N\sum_{k=1}^N
(\bP_{k-1}^{\beta,\omega})^{\otimes 2}
\left(S_k^{(1)}=S_k^{(2)}\right), 
\\ 
-\tf'(\beta)& = \beta \lim_{N\to\infty}
\EE\left[(\bP_N^{\beta,\omega})^{\otimes 2}\left(\frac1N\sum_{k=1}^N \ind_{\{S_k^{(1)}=S_k^{(2)}\}}\right)\right],
\end{align*}
where \((\bP_N^{\beta,\omega})^{\otimes 2}\) denotes the law of two independent trajectories \((S_n^{(1)})_{n\geq 0}\), \((S_n^{(2)})_{n\geq 0}\) drawn under the polymer measure (with a fixed, common environment).
The second identity holds if \(\tf\) is differentiable at \(\beta>0\), which is the case for all but at most countably many \(\beta\), by convexity of \(\beta\mapsto p(\beta) = \tf(\beta) +\lambda(\beta)\); if \(\tf\) is not differentiable, one can bound the \(\liminf\) and \(\limsup\) by the left and right derivatives.

In conclusion, the free energy \(\tf(\beta)\) quantifies how much trajectories are localized, and understanding its critical behavior gives some important information on the phase transition.

\subsection{Main result}

Before we state our result, let us introduce some notation.
For any \(\beta \in [0,\beta_0/2)\) (recall $\beta_0$ from \eqref{hyp:omega}), define
\[
\sigma^2(\beta) =\Var\Big(\e^{\beta \omega_{1,0} -\lambda(\beta)}\Big) = \e^{\lambda(2\beta) -2\lambda(\beta)}-1 \,.
\]
Let us stress right away that we have \(\sigma^2(\beta)\sim \beta^2\) as \(\beta\downarrow 0\), but one may have lower order corrections, see~\eqref{eq:expand-sigma} below.
Our main result gives up-to-constant bounds on the free energy of the two-dimensional directed polymer in the high-temperature limit, expressed in terms of \(\sigma^{2}(\beta)\).
In fact, our contribution is to obtain the lower bound in the following theorem: the upper bound is given in \cite[Theorem~2.8]{BCT25}.

\begin{theorem}
       \label{thm:fe}
       For the two-dimensional directed polymer model, there are two positive constants \(c,c'\) such that, for any \(\beta\in (0,\beta_0/4)\),
       \[
        -c\, \e^{- \frac{\pi}{\sigma^2(\beta)}} \leq \tf(\beta) \leq - c'\, \e^{- \frac{\pi}{\sigma^2(\beta)}} \,.
       \]
\end{theorem}

Since \(\lambda(\beta) = \sum_{n=1}^{\infty} \frac{\kappa_n}{n!} \beta^n\) with \(\kappa_n\) the \(n\)-th cumulant of \(\omega_{1,0}\), we get that 
\begin{equation}
       \label{eq:expand-sigma}
       \sigma^2(\beta) = \beta^2 + \kappa_3 \beta^3 + \frac{7\kappa_4+6}{12} \beta^4 +O(\beta^5) \quad \text{ as } \beta\downarrow 0 \,. 
\end{equation}
Notice that \(\kappa_1=0\), \(\kappa_2=1\) by~\eqref{hyp:omega}, and that \(\kappa_3= \EE[\omega_{1,0}^3]\), \(\kappa_4 = \EE[\omega_{1,0}^4]-3\).
Thus, Theorem~\ref{thm:fe} shows that
\[
-\tf(\beta)  \asymp \e^{-\frac{\pi}{\beta^2} + \frac{\pi}{\beta} \kappa_3}  \quad \text{ as } \beta\downarrow 0 \,,
\]
where \(f(\beta) \asymp g(\beta)\) means that the ratio \(f(\beta)/g(\beta)\) remains bounded away from \(0\) and \(\infty\).

Notice that Theorem~\ref{thm:fe} thus improves Theorem~\ref{thm:existing}-\ref{item-ii}, \textit{i.e.}\ \cite{BL17}, since it removes the \(o(1)\) in the exponential, replacing \(\pi/\beta^2\) by \(\pi/\sigma^2(\beta)\) (which may contain lower order corrections as seen in~\eqref{eq:expand-sigma}).
The upper bound in Theorem~\ref{thm:fe} is proven in \cite[Theorem~2.8]{BCT25}, but our lower bound improves that of \cite[Theorem~2.8]{BCT25} since it differed from the correct asymptotic by some polynomial factor \(\beta^{-8}\).

\begin{remark}
       \label{rem:critical-window}
       Since \(\tf(\beta)\) is the exponential decay rate of the partition function, it is natural to define some \emph{intermediate disorder regime} by letting \(\beta_N\downarrow 0\) in such a way that \(-N\tf(\beta_N) \asymp 1\).
       In other words, one can interpret~\(-1/\tf\) as a correlation length.
       By Theorem~\ref{thm:fe}, the regime \(-N\tf(\beta_N) \asymp 1\) corresponds to \( \sigma^2 (\beta_N) \frac{\log N}{\pi}= 1+ \frac{O(1)}{\log N}\).
       This is exactly the critical window considered to obtain a non-trivial limit of the model, namely the Stochastic Heat Flow of~\cite{CSZ23}. 
       We refer to~\cite{CSZ24-rev} for an introduction and to~\cite{CSZ25} for a recent overview of the literature.
\end{remark}

\subsection{Second moment and Lyapunov exponents}

Let us now give some complementary results. 
First, let us introduce the following Lyapunov exponent:
\[
\tF_2(\beta) = \lim_{N\to\infty} \frac{1}{N} \log \EE\big[(Z_N^{\beta,\omega})^2 \big] \,,
\]
where the existence of the limit is guaranteed by the sub-additivity of \(\log \EE[(Z_N^{\beta,\omega})^2]\), see Section~\ref{sec:Lyapunov}.
In fact, by using a replica trick, since \(\EE\big[\e^{\beta \omega_{n,x} + \beta \omega_{n,y} -2\lambda(\beta)}\big] = \e^{(\lambda(2\beta)-2\lambda(\beta)) \ind_{\{x=y\}}}\) we get that
\[
\EE\big[ (Z_N^{\beta,\omega})^2\big] = \bE^{\otimes 2} \Big[ \e^{ (\lambda(2\beta)-2\lambda(\beta)) \sum_{n=1}^N \ind_{\{S_{n}^{(1)} =S_n^{(2)}\}} }\Big] = \bE \Big[ \e^{ (\lambda(2\beta)-2\lambda(\beta)) \sum_{n=1}^N \ind_{\{S_{2n}=0\}}}\Big] \,.
\]
The last equality above holds by symmetry of the simple random walk.
This is exactly the partition function of a homogeneous pinning model with underlying renewal process \(\tau\) given by the return times of \((S_{2n})_{n\geq 0}\) to \(0\), as studied in \cite[Chap.~2]{Gia07}.
Thus, the Lyapunov exponent \(\tF_2(\beta)\) is the free energy of this pinning model.
Since the renewal process satisfies \(\bP(\tau_1=n) \sim \frac{\pi}{n (\log n)^2}\), see \cite{JP72}, \cite[Theorem~2.1]{Gia07} states that \(\tF_2(\beta)\) vanishes faster than any power.
The following result gives the exact asymptotic behavior of \(\tF_2(\beta)\) in the high-temperature regime; in fact, we provide an asymptotic expansion of the ``correlation length'' \(1/\tF_2(\beta)\). This result follows from classical arguments in \cite[Chap.~2]{Gia07}, but it does not seem to appear in the literature and should serve as a comparison point for Theorem~\ref{thm:fe}.

\begin{theorem}
       \label{thm:second-moment}
       We have the following expansion of the correlation length \(1/\tF_2(\beta)\) in terms of the parameter \(\alpha(\beta) := \pi \frac{1+\sigma^2(\beta)}{\sigma^2(\beta)}\):
       \begin{equation*}
              {\tF_2(\beta)} = {16}\, \e^{-\alpha(\beta)} \Big(1 - 4 \alpha(\beta) \e^{-\alpha(\beta)} + O\big( \alpha(\beta)^2 \e^{-2\alpha(\beta)} \big) \Big)^{-1} \qquad \text{ as } \beta \downarrow 0\,.
       \end{equation*}
       As a consequence, we have \(\tF_2(\beta)  \sim 16\, \e^{-\pi}\, \e^{-\frac{\pi}{\sigma^{2}(\beta)}}\) as \(\beta \downarrow 0\).
\end{theorem}

Note that one can again use the expansion~\eqref{eq:expand-sigma} of \(\sigma^2(\beta)\) to obtain a sharp asymptotic of \(\tF_2(\beta)\) in terms of \(\kappa_3,\kappa_4\).
After some calculations, one gets \(\tF_2(\beta)  \sim 16\, \e^{\pi(\frac{7}{12}\kappa_4-\kappa_3^2 - \frac{1}{2})}\, \e^{-\frac{\pi}{\beta^2} + \frac{\pi}{\beta} \kappa_3} \) as \(\beta \downarrow 0\).

\begin{remark}
       Altogether, combining Theorems~\ref{thm:fe} and \ref{thm:second-moment}, we obtain that
       \[
       -\tf(\beta) \asymp \tF_2(\beta) \quad \text{ as } \beta\downarrow 0\,.
       \]
       In other words, the second moment provides a sharp estimate for the decay rate of the (quenched) partition function.
       In particular, in view of Remark~\ref{rem:critical-window}, the critical window \( \sigma^2 (\beta_N) \frac{\log N}{\pi}= 1+ \frac{O(1)}{\log N}\) for the Stochastic Heat Flow \cite{CSZ23} can also be interpreted as some intermediate disorder regime \(N\tF_2(\beta_N) \asymp 1\) for the second moment. 
       We stress that a similar picture also holds in dimension \(d=1\), where one has \(-\tf(\beta) \asymp \beta^4 \asymp \tF_2(\beta)\) as \(\beta\downarrow 0\), and where the intermediate disorder regime is given by \(N \beta_N^4 \asymp 1\), see \cite{AKQ14a}.
\end{remark}

More generally, one can define the Lyapunov exponent of order \(p>0\) and $\beta \in (0,\beta_0/\max\{1,p\})$, by letting
\[
\tF_p(\beta) = \lim_{N\to\infty} \frac{1}{N} \log \EE\big[(Z_N^{\beta,\omega})^p \big] \,.
\]
The existence of the limit follows by super-additivity if \(p\in (0,1)\) and by sub-additivity if \(p\geq 1\), see Section~\ref{sec:Lyapunov} for details.
We then have the following results, as easy consequences of the above.

\begin{theorem}
       \label{thm:Lyapunov}
       There are universal constants \(c,c'>0\) such that, as \(\beta \downarrow 0\), we have:
       \begin{enumerate}[label=(\roman*)]
              \item \label{i-p-01}
              If \(p\in (0,1)\),
                     \[
                     - c\, p\wedge (1-p)\, \e^{-\frac{\pi}{\sigma^2(\beta)}}  \leq \tF_p(\beta) \leq - c'\, p\wedge (1-p)\, \e^{-\frac{\pi}{\sigma^2(\beta)}} \,.
                     \]
              \item \label{ii-p>1}
              If \(p >1\),
                     \[
                     c'\, (p-1)\, \e^{-\frac{\pi}{\sigma^2(\beta)}}  \leq \tF_p(\beta) \leq  (p-1)\, \tF_{\lceil p \rceil}(\beta) \,.
                     \]
       \end{enumerate}
\end{theorem}

\noindent
In particular, when \(p\in (1,2]\), we have \(\tF_p(\beta) \leq c (p-1)\, \e^{-\frac{\pi}{\sigma^2(\beta)}}\), but we are not aware of any results for the Lyapunov exponent of (integer) order \(p>2\).

\smallskip

Let us end this section by giving some natural conjectures, in view of our results:
\begin{description}
       \item[Conjecture 1] There exists a constant \(c_0 <0\) such that 
       \begin{equation*}
              \tf(\beta) \sim c_0\, \e^{- \frac{\pi}{\sigma^2(\beta)}} \qquad \text{ as } \ \beta\downarrow 0\,.
       \end{equation*}

       \item[Conjecture 2] For any \(p>0\), there exists a constant \(c_p\) with \(c_p<0\) for \(p\in (0,1)\) and \(c_p>0\) for \(p>1\) such that 
       \begin{equation*}
              \tF_p(\beta) \sim c_p\, \e^{- \frac{\pi}{\sigma^2(\beta)}} \qquad \text{ as } \ \beta\downarrow 0\,.
       \end{equation*}
       When \(p>1\) is large, the constant \(c_p\) should be related to the growth of moments inside the critical window: in view of~\cite{GN25} and \cite{Raj99} we thus expect that \(c_p=\exp\big( \e^{\Theta(1) \, p}\big)\); here \(\Theta(1)\) is a quantity that remains bounded away from \(0\) and \(\infty\).
\end{description}

\subsection{Overview of the rest of the paper and of the main steps of the proof}

The rest of the paper is devoted to the proof of the main result, \textit{i.e.}\ the lower bound in Theorem~\ref{thm:fe}. 
The core of the proof lies in Section~\ref{sec:proof}, while a key covariance computation is postponed to Section~\ref{sec:covariances}.
Section~\ref{sec:Lyapunov} contains the proof of the auxiliary results Theorem~\ref{thm:second-moment} and~\ref{thm:Lyapunov}.

Our proof relies on an oriented percolation argument, introduced in \cite{Lac10a} and then refined in~\cite{AY15} in the case of dimension \(d=1\).
The strategy, illustrated in Figure~\ref{fig:oriented-percolation}, is to proceed as follows:
\begin{enumerate}[label=(\roman*)]
       \item Fix some \(N=N(\beta)\) for which the partition function has a small variance: it will thus have a probability close to \(1\) of being larger than some \(c_\star\in (0,1)\);
       \item Perform a coarse-graining of \(\N\times \Z^2\) into time-space boxes of length \(N\) and width \(\lfloor \sqrt{N}\rfloor\), and interpret this as a coarse-grained grid with oriented edges \((n,x) \rightarrow (n+1,x\pm e_i)\) with \(\{e_i\}_{i=1,2}\) the standard coordinate vectors;
       \item Declare an edge ``open'' if the polymer partition function is larger than \(c_\star\) in the corresponding box: a percolation argument shows that there is an infinite open path, with positive probability.
       \item Get a lower bound on \(Z_{mN}^{\beta,\omega}\) by restricting it to some open path, on which it is lower bounded by \((c_\star)^m\). 
\end{enumerate}
This strategy gives the bound \(\liminf_{m\to\infty} \frac{1}{mN} \log Z_{mN}^{\beta,\omega} \geq  \frac{\log c_\star}{N}\), with positive probability.
Since the limit is constant \(\PP\)-a.s.\ and is equal to the free energy, this gives the lower bound \(\tf(\beta) \geq  \frac{\log c_\star}{N}\).

\begin{figure}[tb]
       \centering
       \includegraphics[width=.8\textwidth]{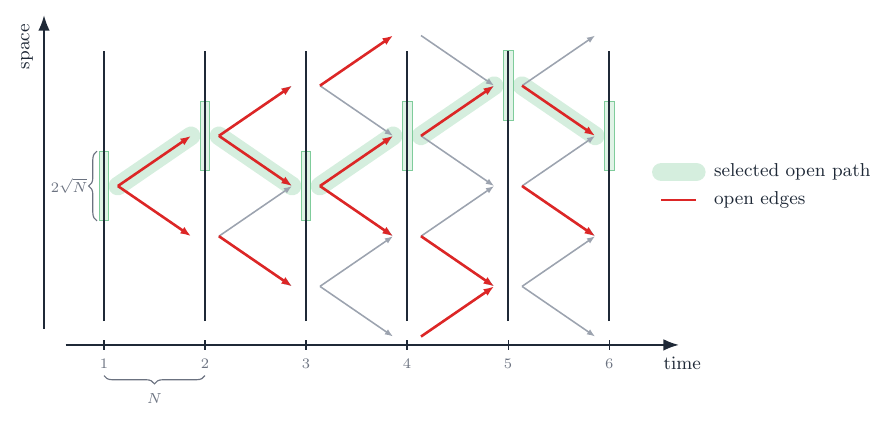}
       \vspace{-15pt}
       \caption{Illustration of the coarse-grained oriented percolation construction. 
       Open edges between boxes are marked with red arrows, while the green boxes highlight a selected open path.}
       \label{fig:oriented-percolation}
\end{figure}

Our goal is thus to use this strategy with \(N=N(\beta) = \gep\, \e^{\pi/\sigma^2(\beta)}\), with some \(\gep\) fixed small enough.
The difficulty here is that the second moment of the point-to-plane partition function~\(Z_N^{\beta,\omega}\) diverges as \(\log N\), see~\cite{CSZ19b}. 
However, the variance becomes small if the initial distribution of the random walk is ``sufficiently diffusive'' (or more precisely sufficiently regular at a diffusive scale), recalling that the regime \(N= \gep\, \e^{\pi/\sigma^2(\beta)}\) corresponds to the critical window for the Stochastic Heat Flow (see e.g.~\cite{CSZ19a} for second moment computations in this critical window).

One key property that we prove is some propagation of ``diffusive regularity'': if one starts with a sufficiently smooth distribution, the end-point distribution is also sufficiently smooth, with high probability. 
We define a ``\(\log\)-energy'' to quantify the diffusive regularity of a distribution, and we show in Section~\ref{sec:log-energy} the propagation of a bounded \(\log\)-energy property, see Proposition~\ref{prop:logregularity-propagation}.
We then embed this property in our definition of an open edge so that, along an open path, all starting and end-point distributions have a bounded \(\log\)-energy: this is detailed in Section~\ref{sec:perco}.
This allows us to conclude the proof of the lower bound of Theorem~\ref{thm:fe}, in Section~\ref{sec:conclusion}.

\section{Proof of the lower bound in Theorem~\ref{thm:fe}}
\label{sec:proof}

We assume that $d=2$ and $\beta \in (0,\beta_0/4)$ from now on.

\subsection{First notations}
\label{sec:notations}

Set \(R_N:=\sum_{n=1}^{N}\bP(S_{2n}=0)\) and note that \cite[Proposition~3.2]{CSZ19b} shows that \(R_N = \frac{1}{\pi}(\log N - \alpha +o(1))\), where \(\alpha=4\log 2+\gamma-\pi\) with \(\gamma\) the Euler--Mascheroni constant.

We fix some \(\vartheta \in \R\), which will be taken sufficiently negative, and we consider 
\begin{equation}
       \label{def:N}
       N = N(\beta) := \left\lfloor \e^{\vartheta +\alpha} \e^{\pi/\sigma^2(\beta)} \right\rfloor \,.
\end{equation}
With this notation, we get that, as \(\beta \downarrow 0\) (or equivalently \(N\to\infty\)),
\[
       \sigma^2(\beta)R_N = \frac{\sigma^2(\beta)}{\pi} \Big( \vartheta + \frac{\pi}{\sigma^2(\beta)     } + o(1) \Big)= 1+ \frac{\vartheta+o(1)}{\log N} \,.
\]
In the following, we will fix some \(\vartheta =\vartheta_0<0\) (negative enough) and we will assume that~\(\beta\) is small enough (or equivalently \(N\) large enough) so that
\begin{equation}
       \label{def:critical-window}
       \sigma^2(\beta)R_N \leq 1+ \frac{\vartheta_0+1/2}{\log N} \,.
\end{equation}
To light notation, we will also omit integer parts and write \(\sqrt{N}\) for $\lfloor \sqrt N \rfloor$. 

For \(u,z\in\Z^2\), we use the notation
\[
Z_N^{\beta}(u,z)
:= \bE_u\left[ \e^{\sum_{k=1}^{N}(\beta\omega_{k,S_k}-\lambda(\beta))} \ind_{\{S_N=z\}} \right]
\]
for the point-to-point partition function, and for some probability measure \(\mu\) on \(\Z^2\) and \(B\subset \Z^2\), we also write 
\[
Z_N^{\beta}(\mu,B) := \sum_{u\in \Z^2, z\in B} \mu(u) \, Z_N^{\beta}(u,z) \,.
\]
We naturally define similarly \(Z_N^{\beta}(u,B)\) and \(Z_N^{\beta}(\mu,z)\). (With this notation \(Z_N^{\beta,\omega} = Z_N^{\beta}(0,\Z^2)\).)
We also introduce a constrained version of the partition functions:
\[
Z_N^{\beta,\cons}(u,z) := \bE_u\left[ \e^{\sum_{k=1}^{N}(\beta\omega_{k,S_k}-\lambda(\beta))} \ind_{\{S_N=z\}}  \ind_{\{\max_{1\leq k\leq N}|S_k|_\infty < 5 \sqrt{N}\}}\right] \,,
\]
and similarly for \(Z_N^{\beta,\cons}(\mu,B)\).


Let \(e_1,e_2\) be the standard coordinate vectors and let us set
\[
\Lambda_N := (-\sqrt{N}, \sqrt{N}]^2 \cap \Z^2,\quad
\Lambda_N^a:=2a \sqrt{N}+\Lambda_N,\quad a\in\{\pm e_i\}_{i=1,2} \,.
\]
We denote by \(\mathcal{M}_1(B)\) the set of probability distributions with support in \(B\).

\begin{lemma}
       \label{lem:concentration}
       There exists a constant \(c_{\star}>0\) such that, for any $a\in\{\pm e_i\}_{i=1,2}$, for any \(N\geq 16\),
       \[
       1\geq \inf_{\mu \in \mathcal{M}_1(\Lambda_N)} \EE\big[Z_N^{\beta,\cons}(\mu,\Lambda_N^a)\big] \geq c_{\star} \,.
       \]
\end{lemma}

\begin{proof}    
The point-to-point partition functions are nonnegative by definition.  
Moreover, the normalization $\mathbb E[\exp(\beta\omega-\lambda(\beta))]=1$ gives, for any $a\in\{\pm e_i\}_{i=1,2}$,
$$
\mathbb{E}[Z_N^{\beta,\cons}(u,\Lambda_N^{a})] = \bP_u\Big(S_N\in \Lambda_N^{a},\,\max_{1\le k\le N}|S_k|_\infty< 5\sqrt{N}\Big) \,.
$$
By the invariance principle, we get that there is a positive constant \(c\) such that, for all \(N\) large enough, $1\geq \mathbb{E}[Z_N^{\beta,\cons}(u,\Lambda_N^{a})]\geq c$, uniformly for \(u\in \Lambda_N\) and $a\in\{\pm e_i\}_{i=1,2}$. This directly proves the lemma, adjusting the constant to \(c_{\star}\) to cover the case \(N\geq 16\) (the restriction~\(N\geq 16\) simply prevents that \(\bP_u(S_N\in \Lambda_N^a)=0\)).
\end{proof}

\subsection{The bounded \texorpdfstring{\(\log\)}{log}-energy property}
\label{sec:log-energy}

Given a non-negative kernel \(K:\Z^2 \to \R_+\) and a probability measure \(\mu\) on \(\Z^2\), one may define the \(K\)-energy of \(\mu\) as follows:
\begin{equation}
       \label{def:energy}
       \mathcal{E}_{K}(\mu) := \sum_{u,v \in \Z^2} \mu(u)K(v-u)\mu(v)\,.
\end{equation}
In the rest of the article, we consider the following log-kernel which acts at a diffusive scale:
\begin{equation}
       \label{def:K-N}
     {  K_N(x) := \Big(1+\log_+\Big(\frac{N}{1+|x|^2}\Big)\Big) \, \e^{ - \frac{|x|^2}{2N}} \,.}
\end{equation}

\begin{remark}
       The kernel that arises naturally in covariance computations is actually \(Q_N(x):= \sum_{n=1}^N \bP(S_{2n}=x)\), which verifies \(Q_N(x) \leq C\, K_N(x)\) for some constant \(C>0\), uniformly in \(N\geq 1\), \(x\in \Z^2\), see Remark~\ref{rem:compare-K-Q}.
       However, our estimates on the covariances involve \(K_N\) rather than \(Q_N\), see Lemma~\ref{lem:covariance} below --- this is mostly to allow for a simpler proof.
\end{remark}


Our next result shows that, provided that the initial distribution has a bounded \(\log\)-energy, so does the endpoint distribution, with high probability.
To state the result, we introduce, for \(\mu \in \mathcal{M}_1(\Lambda_N)\), the corresponding endpoint distribution conditioned on \(\Lambda_N^a\), \(a\in \{\pm e_i\}_{i=1,2}\): 
\begin{equation}
       \label{def:pol-measure}
       \nu_N^{\beta,a}(\mu,z) :=  \frac{Z_N^{\beta,\cons}(\mu,z)}{Z_N^{\beta,\cons}(\mu,\Lambda_N^a)} \qquad \text{ for } z\in \Lambda_N^a \,.
\end{equation}
All the estimates below are uniform in \(a\in \{\pm e_i\}_{i=1,2}\).

\begin{proposition}
\label{prop:logregularity-propagation}
Let \(c_{\star}>0\) be the constant in Lemma~\ref{lem:concentration}.
For any $a\in\{\pm e_i\}_{i=1,2}$, for every $\delta>0$, there exist $A>1$ and $\vartheta_0<0$ such that, if  $N$ satisfies~\eqref{def:critical-window} and is sufficiently large, then for any $\mu \in \mathcal{M}_1(\Lambda_N)$,
$$
\mathcal{E}_{K_N}(\mu) \leq A 
\quad \Longrightarrow \quad 
\mathbb{P}\left(
Z_N^{\beta,\cons}(\mu,\Lambda_N^a)\ge c_\star/2,\ 
\mathcal{E}_{K_N}\big(\nu_N^{\beta,a}(\mu,\cdot) \big) \leq A
\right) \ge 1-\delta.
$$
\end{proposition}

We stress that the \(\log\)-energy gives some indication of the ``regularity'' of a measure \(\mu\) on \(\Z^2\).
For instance, for any probability measure \(\mu\) on \(\Z^2\), letting \(B_R\) denote a ball of radius \(R\leq \sqrt{N}\),
\[
\sum_{|u|,|v| \leq R} \mu(u)\mu(v) K_N(v-u) \geq c\,\mu(B_R)^2 \log\Big(1+\frac{N}{1+R^2}\Big) \,,
\]
for some universal constant \(c>0\).
In particular, a probability measure \(\mu\) with bounded \(\log\)-energy cannot have a mass larger than \(C \log(1+N/(1+R^2))^{-1/2}\) on balls of radius \(R\leq \sqrt{N}\).
Thus, Proposition~\ref{prop:logregularity-propagation} provides some insight into the ``diffusive regularity'' for the polymer measure when started from a regular enough measure. 
We have kept our result quite rudimentary in order to have a simpler proof, but one could prove an analogous result for the Stochastic Heat Flow (SHF) of \cite{CSZ23}, to show that the SHF measures have a finite \(\log\)-energy --- this is the object of some upcoming work by Caravenna, Sun and Zygouras~\cite{CSZ-next}, where they obtain much stronger and more detailed results in this direction (see also the discussion in \cite[\S10.2]{CSZ24-rev}).

\begin{remark}[\(\log\)-energy of the uniform measure]
       \label{rem:energy-unif}
       Let us stress that a simple computation shows that the uniform measure on~\(\Lambda_N\) has a bounded \(\log\)-energy; in fact, it has a finite ``\(\log^2\)-energy''.
       In other words, there are some universal constants \(H_1,H_2\) such that, for all \(N\geq 1\),
       \[
        \frac{1}{N^2} \sum_{u,v \in \Lambda_N} K_N(v-u)  \leq H_1 \,,\qquad \frac{1}{N^2} \sum_{u,v \in \Lambda_N} K_N(v-u)^2  \leq H_2 \,.
       \]
\end{remark}

The proof of Proposition~\ref{prop:logregularity-propagation} will strongly rely on some second moment estimates, so let us state here a useful lemma, whose proof is postponed to Section~\ref{sec:covariances}.

\begin{lemma}
       \label{lem:covariance}
       First of all, for any \(u,z,v,w \in \Z^2\),
       $$
       \Cov\bigl(Z_N^{\beta,\cons}(u,z),Z_N^{\beta,\cons}(v,w)\bigr)
       \le \Cov\bigl(Z_N^{\beta}(u,z),Z_N^{\beta}(v,w)\bigr) \,.
       $$
       Additionally, one has a ``factorized'' upper bound: there is a universal constant \(C\) such that, for all sufficiently large \(N\) satisfying~\eqref{def:critical-window} and all $u,v,z,w\in\Z^2$
       \begin{equation}
              \label{ineq:covariances}
              \Cov\bigl(Z_N^{\beta}(u,z),Z_N^{\beta}(v,w)\bigr)
              \le \frac{C \,\mathcal{V}(\e^{\vartheta_0})}{ N^2} \, K_N(v-u) \, K_N(w-z)\,,
       \end{equation}
       where \(\mathcal{V}(t):=\int_0^\infty \frac{t^s}{\Gamma(s+1)}\,\dd s\) is the Volterra function; we stress that \(\mathcal{V}(t) \sim \frac{1}{|\log t|}\) as \(t\downarrow 0\).
\end{lemma}

\begin{remark}
       We expect that one could in fact show that 
       \[
       \Cov\bigl(Z_N^{\beta}(u,z),Z_N^{\beta}(v,w)\bigr)
              \le \frac{C \,\mathcal{V}(\e^{\vartheta_0})}{ N^2} \, Q_N(v-u) \, Q_N(w-z) \times \exp\Big(- c \,\frac{|\frac{z+w}{2} -\frac{u+v}{2}|^2}{N} \Big) \,,
       \]
       where \(Q_N(x):= \sum_{n=1}^N \bP(S_{2n}=x)\) is the ``natural'' kernel appearing in the covariance computations (recall that we have \(Q_N(x)\leq C K_N(x)\)).
       Note that the covariances should also decay if the starting ``mid-points'' \(\frac{z+w}{2}\) and \(\frac{u+v}{2}\) are far apart (at a super-diffusive scale). 
       This should follow from arguments similar but more technical than our proof, and we have chosen not to include this in Lemma~\ref{lem:covariance} for simplicity, since we do not need such a refined result.
\end{remark}

\begin{proof}[Proof of Proposition~\ref{prop:logregularity-propagation}]
Fix $\delta \in (0,1)$, and let \(A:=C_0 \delta^{-1}\) for some constant \(C_0\) large enough (that we determine explicitly, see below).
Then, we write 
\begin{equation}
       \label{eq:two-terms}
       \begin{split}
       \mathbb{P}\Big(&Z_N^{\beta,\cons}(\mu,\Lambda_N^a)< c_\star/2  \text{ or } \mathcal{E}_{K_N}\big(\nu_N^{\beta,a}(\mu,\cdot) \big) > A \Big)  \\
       &\qquad \leq \mathbb{P}\big(Z_N^{\beta,\cons}(\mu,\Lambda_N^a)< c_\star/2\big) + \mathbb{P}\big(Z_N^{\beta,\cons}(\mu,\Lambda_N^a)\geq c_\star/2 \,, \mathcal{E}_{K_N}\big(\nu_N^{\beta,a}(\mu,\cdot) \big) > A \big) \,,    
\end{split}
\end{equation}
and we estimate both terms separately.

\smallskip
\noindent
\textit{First term in~\eqref{eq:two-terms}: concentration of the partition function.}
For \(\mu\in \mathcal{M}_1(\Lambda_N)\), let us write
$$
V_N^{\beta}(\mu) := \Var\left(Z_N^{\beta,\cons}(\mu,\Lambda_N^a)\right)
= \sum_{u,v\in\Lambda_N} \sum_{z,w\in\Lambda_N^a} \mu(u)\mu(v) \, \Cov\bigl(Z_N^{\beta,\cons}(u,z),Z_N^{\beta,\cons}(v,w)\bigr) \,.
$$
Then, Lemma~\ref{lem:covariance} gives that 
\[
V_N^{\beta}(\mu) 
\le C \,\mathcal{V}(\e^{\vartheta_0}) 
\sum_{u,v\in\Lambda_N}\mu(u)\mu(v)K_N(v-u)
\frac{1}{ N^2}  \sum_{z,w\in\Lambda_N^a}K_N(w-z) \le
C \,\mathcal{V}(\e^{\vartheta_0})\, \mathcal{E}_{K_N}(\mu) \,H_1 \,,
\]
with \(H_1\) the constant from Remark~\ref{rem:energy-unif}.
Then, since \(\mathcal{E}_{K_N}(\mu) \leq A\) and \(\mathcal{V}(t) \to 0\) as \(t\downarrow0\), we can choose \(\vartheta_0 = \vartheta(A,\delta)<0\) sufficiently negative so that \(V_N^{\beta}(\mu) \leq  c_{\star}^2 \delta /8\).

Now, recalling that \(\mathbb{E}[Z_N^{\beta,\cons}(\mu,\Lambda_N^a)] \geq c_{\star}\) by Lemma~\ref{lem:concentration}, we get by Chebyshev's inequality that the first term in \eqref{eq:two-terms} is 
\[
       \mathbb{P}\left(Z_N^{\beta,\cons}(\mu,\Lambda_N^a)<c_\star/2\right) \leq \frac{4}{c_\star^2}\, V_N^{\beta}(\mu) \le \frac{\delta}{2} \,,
\]
uniformly for \(\mu\in \mathcal{M}_1(\Lambda_N)\) with \(\mathcal{E}_{K_N}(\mu) \leq A\).

\smallskip
\noindent
\textit{Second term in~\eqref{eq:two-terms}: \(\log\)-energy of the polymer measure.}
Recall the definition~\eqref{def:pol-measure}, so that 
\[
\mathcal{E}_{K_N}\big(\nu_N^{\beta,a}(\mu,\cdot) \big) =  \frac{1}{Z_N^{\beta,\cons}(\mu,\Lambda_N^a)^2} \sum_{z,w\in \Lambda_N^a} Z_N^{\beta,\cons}(\mu,z)Z_N^{\beta,\cons}(\mu,w) \, K_N(w-z) \,.
\]
Then, on the event $\{Z_N^{\beta,\cons}(\mu,\Lambda_N^a)\ge c_\star/2\}$, we have that \(\mathcal{E}_{K_N}(\nu)\le 4 \mathcal R/c_\star^2 \) with 
\[
\mathcal R := \sum_{z,w\in \Lambda_N^a} Z_N^{\beta}(\mu,z)Z_N^{\beta}(\mu,w) \, K_N(w-z) \,,
\]
where we have further removed the constraint in the partition function to obtain an upper bound.
We thus have that the second term in~\eqref{eq:two-terms} is bounded by 
\[
       \mathbb{P}\big( \mathcal R > A c_\star^2/4\big) \leq \frac{4}{c_{\star}^2 A}\, \EE[\mathcal R] \,, 
\]
by Markov's inequality.
Now, let us note that
$$
\mathbb{E}[Z_N^{\beta}(\mu,z)Z_N^{\beta}(\mu,w)]
= \mathbb{E}[Z_N^{\beta}(\mu,z)] \mathbb{E}[Z_N^{\beta}(\mu,w)] +\Cov(Z_N^{\beta}(\mu,z),Z_N^{\beta}(\mu,w)) \,.
$$
Then, noting that \(\mathbb{E}[Z_N^{\beta}(\mu,z)] = \bP_{\mu}(S_N =z) \leq \frac{C}{N}\) by the local-CLT (see e.g.\ \cite[Prop.~2.4.4]{LL10}) and using also Lemma~\ref{lem:covariance} to control the covariances, we get that 
\[     
\EE[\mathcal R] \leq \frac{C^2}{N^2}  \sum_{z,w\in \Lambda_N^a} K_N(w-z) + C \mathcal{V}(\e^{\vartheta_0}) \sum_{u,v \in \Lambda_N} \mu(u)\mu(v) K_N(v-u)  \frac{1}{N^2} \sum_{z,w\in \Lambda_N^a} K_N(w-z)^2 \,.
\]
Recalling that \(\mathcal{E}_{K_N}(\mu)\leq A\) and Remark~\ref{rem:energy-unif} for the definition of the constants \(H_1,H_2\), we get
\[
\mathbb{P}\big( \mathcal R > A c_\star^2/4\big) \leq  \frac{4}{c_{\star}^2 A} \Big(C^2 H_1 + C \mathcal{V}(\e^{\vartheta_0}) A H_2 \Big) \leq \frac{4 C^2 H_1}{ c_{\star}^2} A^{-1} + \frac{4 C H_2 }{ c_{\star}^2}\, \mathcal{V}(\e^{\vartheta_0}) \,.
\]
Therefore, if we have set \(A = C_0 \delta^{-1}\) with \(C_0 = 16 C^2 H_1/c_{\star}^2\) and if \(\vartheta_0\) is negative enough so that \(\mathcal{V}(\e^{\vartheta_0})  \leq c_{\star}^2\delta/(16 C H_2)\), we conclude that the second term in~\eqref{eq:two-terms} is bounded by~\(\delta/2\), again uniformly for \(\mu\in \mathcal{M}_1(\Lambda_N)\) with \(\mathcal{E}_{K_N}(\mu) \leq A\).
This concludes the proof.
\end{proof}

\subsection{Implementing the percolation argument}
\label{sec:perco}

For $x\in\mathbb Z^2$, we define the space-shifted box
$$
\Lambda_N^{x}:=2x\sqrt N+\Lambda_N \,.
$$
We consider oriented edges of the form \(e = ((k,x),(k+1,y))\) with \(|x-y|_1=1\).
We will define a percolation process on edges iteratively in time \(k\geq 1\).
We start at time \(k=1\) and \(x=0\) for simplicity (the reason will be clear in Section~\ref{sec:conclusion} below), and let us introduce a notation for the set of edges that we consider at time~\(k\geq 1\):
\[
E_k := \big\{ e= ((k,x),(k+1,y)), \, |x|_1\leq k-1, |x-y|_1 =1\big\} \,,\qquad E =\bigcup_{k\geq 1} E_k \,.
\]
We fix \(\delta\in (0,1)\) and choose \(A \geq 1\), \(\vartheta_0<0\), and \(N\) in~\eqref{def:critical-window} so that Proposition~\ref{prop:logregularity-propagation} holds  and so that \(A\) is larger than the \(\log\)-energy of the uniform measures on \(\Lambda_N^y\).

The goal is to define probability measures \((\mu_{k,x})_{k\geq 1, |x|_1< k}\) on \(\Lambda_N^x\) such that the \(\mu_{k,x}\) are \(\mathcal{F}_{kN}\)-measurable, with \(\mathcal{F}_n = \sigma\{ \omega_{i,z}, i\leq n, z \in \Z^2\}\), and such that \(\mathcal{E}_{K_N}(\mu_{k,x}) \leq A\).
We will also construct sets \((V_k^*)_{k\geq 0}\) of ``connected'' vertices at time \(k\), \textit{i.e.}\ vertices that are connected to \((1,0)\) by (at least) one oriented path of open edges.

For \(u\in \Lambda_N^x\) and \(0\le a<b\), we use the time-shifted notation:
\[
Z_{(a,b]}^{\beta,\cons}(u,z)
:= \bE_u\Big[ \e^{\sum_{j=1}^{b-a}(\beta\omega_{a+j,S_j}-\lambda(\beta))} \ind_{\{S_{b-a}=z\}} \ind_{\{ \max_{1\leq j \leq b-a} |S_j - 2x \sqrt{N} |_\infty < 5 \sqrt{N}\}} \Big] \,,
\]
and similarly for \(Z_{(a,b]}^{\beta,\cons}(\mu,B)\) for probability measures and target sets, in analogy with the notation of Section~\ref{sec:notations}.
Then, for oriented edges \(e=((k,x),(k+1,y)) \in E_k\) and any probability measure \(\mu \in \mathcal{M}_1(\Lambda_N^x)\), we define 
\[
\nu_{N}^{\beta,e}(\mu,z)
:=\frac{Z_{(kN,(k+1)N]}^{\beta,\cons}(\mu,z)}{Z_{(kN,(k+1)N]}^{\beta,\cons}(\mu,\Lambda_N^y)} \, ,
\qquad z\in\Lambda_N^y \,,
\]
in analogy with~\eqref{def:pol-measure}.
Then, we define the event that the oriented edge \(e=((k,x),(k+1,y)) \in E_k\) is \emph{open} as follows: 
\begin{equation}
       \label{def:open}
       \{e \text{ is open}\} = \Big\{
       Z_{(kN,(k+1)N]}^{\beta,\cons}(\mu_{k,x},\Lambda_N^{y})\geq c_\star/2
       \ \text{ and }\ 
       \mathcal{E}_{K_N}(\nu_N^{\beta,e}(\mu_{k,x},\cdot)) \leq A \Big\} \,.  
\end{equation}

To construct iteratively the measures \(\mu_{k,x}\), we start with an initial \(\mu_{1,0} \in \mathcal{M}_1(\Lambda_N^0)\) with \(\mathcal{E}_{K_N}(\mu_{1,0}) \leq A\).
(Some specific choice for \(\mu_{1,0}\) is made in Section~\ref{sec:conclusion} below.)
We also select the first vertex in the set of connected vertices at time \(k=1\), \textit{i.e.}\ we let \(V_1^*=\{0\}\).

Then, assume that we have defined at time \(k\) some $\mathcal F_{kN}$-measurable probability measures $\mu_{k,x}$ with \(\mathcal{E}_{K_N}(\mu_{k,x})\leq A\) and the set \(V_k^*\) of connected vertices.
Using the definition of an open edge, we now define the set of connected vertices at time \(k+1\) as follows:
\[
V_{k+1}^* = \big\{ y \in \Z^2 \,, \exists x \in V_k^* \text{ such that $e = ((k,x),(k+1,y))$ is open} \big\} \,.
\]
As far as the measures \(\mu_{k+1,y}\) are concerned, we define them as follows:
\begin{itemize}
       \item If \(y \notin V_{k+1}^*\), we let \(\mu_{k+1,y} = U_{\Lambda_N^y}\), where \(U_{\Lambda_N^y}(z) = \frac{1}{|\Lambda_N^y|} \ind_{\Lambda_N^y}(z)\) is the uniform measure on \(\Lambda_{N}^y\) (in fact, any measure with \(\log\)-energy bounded by \(A\) would be fine);
       \item If \(y \in V_{k+1}^*\), we select with a deterministic rule some \(x\in V_k^*\) such that $e = ((k,x),(k+1,y))$ is open (say the smallest \(x\) for the lexicographic order), and we set \(\mu_{k+1,y}(\cdot) = \nu_N^{\beta,e}(\mu_{k,x},\cdot)\).
\end{itemize}
Note that \(\mu_{k+1,y}\) is constructed in an \(\mathcal{F}_{(k+1)N}\)-measurable way.

We have thus constructed a percolation process on oriented edges \(e \in E\).
The next lemma shows that if \(\delta\) has been fixed sufficiently small, then there is an infinite oriented open path starting from \((1,0)\) with positive probability.

\begin{lemma}
       \label{lem:percolation}
       If \(\delta>0\) is fixed small enough and \(A,\vartheta_0\) are chosen such that the conclusion of Proposition~\ref{prop:logregularity-propagation} holds, then
       $$
       \mathbb{P}\big((1,0)\longrightarrow \infty
       \text{ by open oriented edges}\big)\ge \frac12 \, ,
       $$
       where $\{(1,0)\longrightarrow \infty \text{ by open oriented edges}\}$ denotes the event that there exists an infinite oriented open path starting from $(1,0)$.
\end{lemma}

\begin{remark}
       \label{rem:infinite-path}
       Let us stress that, by construction, if there exists some infinite path of open oriented edges, one can select some infinite sequence \(e_k = ((k,x_k),(k+1,x_{k+1}))\) of open edges such that \(x_k \in V_k^*\) and \(\mu_{k+1,x_{k+1}}(\cdot) = \nu_N^{\beta,e_k}(\mu_{k,x_k},\cdot)\) for all \(k\geq 1\).
       Indeed, the finite open paths compatible with the deterministic predecessor rule form a locally finite infinite tree, and K\"onig's lemma gives an infinite branch.
\end{remark}

\begin{proof}
       First of all, notice that by construction, all \((\mu_{k,x})_{k\geq 1, |x|_1<k}\) have a \(\log\)-energy bounded by~\(A\).       Thus, applying Proposition~\ref{prop:logregularity-propagation}, we get that
       \[
       \mathbb P\left(e\text{ is open} \mid \mathcal F_{kN}\right) \ge 1-\delta
       \qquad\text{for all } e\in E_k.
       \]
       Note also that thanks to the constraint in the partition function in the definition~\eqref{def:open} of an open edge, conditionally on \(\mathcal F_{kN}\), the indicators \((\ind_{\{e \text{ is open}\}})_{e\in E_k}\) have only finite-range space dependence \(r=10\).  
       Indeed, a constrained path starting from~\(\Lambda_N^x\) remains contained in \(2x\sqrt N+(-5\sqrt N,5\sqrt N]^2\), so the partition functions associated to two edges \(e=((k,x),(k+1,y))\), \(e=((k,x'),(k+1,y'))\) are independent as soon as \(|x-x'| >10\).
       Thus, by the Liggett--Schonmann--Stacey domination theorem~\cite{LSS97}, conditionally on \(\mathcal{F}_{kN}\), the indicators \((\ind_{\{e \text{ is open}\}})_{e\in E_k}\) stochastically dominate independent Bernoulli variables \((\eta_e)_{e\in E_k}\) with parameter \(p:=1-\delta_{\mathrm{LSS}}\), where the parameter \(\delta_{\mathrm{LSS}}\) can be made arbitrarily small by choosing \(\delta\) small.

       We thus have a conditional coupling such that \(\eta_e\le 1_{\{e\text{ is open}\}}\) for all \(e\in E_k\).
       Applying these conditional coupling kernels successively, with independent auxiliary randomness at each time~\(k\), we obtain a joint coupling such that \(\eta_e\le 1_{\{e\text{ is open}\}}\) for all \(e\in E\), where the \((\eta_e)_{e\in E}\) are i.i.d.\ Bernoulli variables of parameter \(p\).
       We thus have that the probability that there is an infinite oriented open path starting from \((1,0)\) is bounded from below by \(\theta(p)\), where \(\theta(p)\) is the survival probability for i.i.d.\ \(\mathrm{Bern}(p)\) oriented bond percolation.
       Since \(\theta(p) \uparrow 1\) as \(p\uparrow 1\), see~\cite{Durrett84}, this concludes the proof, provided that \(\delta\) has been fixed small enough so that \(p=1-\delta_{\mathrm{LSS}}\) is close enough to \(1\).
\end{proof}

\subsection{Conclusion of the proof}
\label{sec:conclusion}

{Let us define \(\tilde\Lambda_N^0 :=\{ x \in \Lambda_N^0, |x|_1+N \text{ is even}\}\).
We let the initial probability measure \(\mu_{1,0}\) be the uniform measure on \(\tilde\Lambda_N^0\), and we choose \(A_0>1\) large enough so that \(\mathcal{E}_{K_N}(\mu_{1,0})\leq A_0\). We fix \(\delta \in (0,1)\), \(A>A_0,\vartheta_0<0\) such that Proposition~\ref{prop:logregularity-propagation} and Lemma~\ref{lem:percolation} hold.
Recall that we have defined \(N = N(\beta) := \lfloor \e^{\vartheta_0 +\alpha} \e^{\pi/\sigma^2(\beta)}\rfloor\) in~\eqref{def:N}, and we take $\beta>0$ small enough so that \(N\) satisfies~\eqref{def:critical-window} and sufficiently large.}

Then, since $Z_N^{\beta}(0,z)>0$ a.s.\ for every $z\in\tilde\Lambda_{N}^{0}$ and since \(\tilde\Lambda_{N}^{0}\) is finite, there exists some $\eta_N\in(0,1)$ (whose specific value is immaterial) such that, setting
$$
\mathcal A_1:= \left\{Z_N^{\beta}(0,z)\ge \eta_N\mu_{1,0}(z)\text{ for all }
z\in\tilde\Lambda_N^{0}\right\} \,,
$$
we have \(\mathbb P(\mathcal A_1)\ge  \frac12\). Let also 
\[
\mathcal A_2:=\{(1,0)\longrightarrow \infty \text{ by open oriented edges}\} \,,
\]
with initial measure \(\mu_{1,0}\) as above.
Then, Lemma~\ref{lem:percolation} shows that $\mathbb P(\mathcal A_2 )\ge 1/2$. 
Moreover, since $\mathcal A_1 \in \mathcal{F}_N$ while $\mathcal A_2$ is independent of \(\mathcal{F}_N\), we get that $\mathbb P(\mathcal A_1\cap\mathcal A_2)\ge 1/4$.

\smallskip
Let us now work on the event $\mathcal A_1\cap\mathcal A_2$, and consider an infinite path of open oriented edges \(e_k = ((k,x_k),(k+1,x_{k+1}))\), with \(x_1=0\), and associated measures \((\mu_{k,x_k})_{k\geq 1}\) verifying \(\mu_{k+1,x_{k+1}}(\cdot) = \nu_N^{\beta,e_k}(\mu_{k,x_k},\cdot)\) for all \(k\geq 1\) (recall Remark~\ref{rem:infinite-path}).
Then, by restricting the path to follow the boxes \((\Lambda_N^{x_k})_{k\geq 1}\), we get by the Markov property that, for any \(m\geq 1\),
\begin{align*}
 Z_{mN}^{\beta}
&\ge \sum_{z_1\in \Lambda_N^{x_1},z_2\in\Lambda_N^{x_2},\ldots,z_m\in\Lambda_N^{x_m}}
Z_{(0,N]}^{\beta}(0,z_1)
\prod_{k=1}^{m-1} Z_{(kN,(k+1)N]}^{\beta,\cons}(z_k,z_{k+1})\\
&\ge \eta_N
\sum_{z_2\in\Lambda_N^{x_2},\ldots,z_m\in\Lambda_N^{x_m}}
Z_{(N,2N]}^{\beta,\cons}(\mu_{1,0},z_2)
\prod_{k=2}^{m-1} Z_{(kN,(k+1)N]}^{\beta,\cons}(z_k,z_{k+1}).
\end{align*}
where we have used the event \(\mathcal{A}_1\) for the first time interval.

Now, for every open edge $e_k=((k,x_k),(k+1,x_{k+1}))$ on this path and every $z\in\Lambda_N^{x_{k+1}}$, we have that
$$
Z_{(kN,(k+1)N]}^{\beta,\cons}(\mu_{k,x_k},z)
=Z_{(kN,(k+1)N]}^{\beta,\cons}(\mu_{k,x_k},\Lambda_N^{x_{k+1}}) \, \mu_{k+1,x_{k+1}}(z)
\ge \frac{c_\star}{2}\, \mu_{k+1,x_{k+1}}(z) \,,
$$
recalling also that \(Z_{(kN,(k+1)N]}^{\beta,\cons}(\mu_{k,x_k},\Lambda_N^{x_{k+1}}) \geq c_{\star}/2\) on an open edge, see~\eqref{def:open}.
Applying this identity successively yields that \(Z_{mN}^{\beta} \ge \eta_N (c_{\star}/2)^{m-1}\) on the event $\mathcal A_1\cap \mathcal A_2$.
In particular, on this event we get that 
$$
\liminf_{m\to\infty} \frac{1}{mN} \log Z_{mN}^{\beta} \geq \liminf_{m\to\infty} \frac{1}{m} \frac{\log \eta_N}{ N} + \liminf_{m\to\infty} \frac{m-1}{m} \frac{1}{N}\log (c_{\star}/2) =\frac{1}{N}\log (c_{\star}/2) \,.
$$
Since $\tf(\beta) = \lim_{m\to\infty} \frac{1}{mN} \log Z_{mN}^{\beta}$ is deterministic and $\mathbb P(\mathcal A_1\cap \mathcal A_2)\geq 1/4$, the almost sure limit satisfies
$$
\tf(\beta) \geq -\frac{\log(2/c_\star)}{N}.
$$
Recalling the definition~\eqref{def:N} of $N$ (with \(\vartheta =\vartheta_0\) so that Proposition~\ref{prop:logregularity-propagation} and Lemma~\ref{lem:percolation}), this gives, for all sufficiently small $\beta$,
\[
\tf(\beta)\ge -C_{\vartheta_0} \,\e^{-\pi/\sigma^2(\beta)},
\qquad
C_{\vartheta_0}:=2\e^{-(\vartheta_0+\alpha)}\log(2/c_\star).
\]
The lower bound in Theorem~\ref{thm:fe} then holds for all \(\beta\in (0,\beta_0/4)\) simply by enlarging the constant (recall that \(\tf\) and \(\sigma^2\) are continuous).
\qed

\section{Covariance estimates: proof of Lemma~\ref{lem:covariance}}
\label{sec:covariances}

\subsection{Preliminaries: chaos expansion formulas for covariances}

We recall here the classical chaos expansion of the partition function.
For \(n\in \N,x\in \Z^d\), let \(\xi_{n,x} := \e^{\beta \omega_{n,x}-\lambda(\beta)}-1\), so that \(\EE[\xi_{n,x}]=0\) and \(\EE[\xi_{n,x}^2] = \sigma^2(\beta) = \e^{\lambda(2\beta)-2\lambda(\beta)}-1\). Thus, one can write 
\[
Z_N^{\beta}(u,z) = \e^{\beta \omega_{N,z} -\lambda(\beta)} \bE_u\Big[ \prod_{n=1}^{N-1}  (1+\xi_{n,S_n}) \ind_{\{S_{N}=z\}} \Big]  =: \e^{\beta \omega_{N,z} -\lambda(\beta)} \tilde Z_N^{\beta}(u,z)\,.
\]
Note that we have, for \(z\neq w\),
\begin{equation}
       \label{eq:Z-tilde-Z}
       \Cov(Z_N^{\beta}(u,z), Z_N^{\beta}(v,w)) =  \Cov(\tilde Z_N^{\beta}(u,z), \tilde Z_N^{\beta}(v,w))\,,
\end{equation}
while for \(z=w\)
\begin{equation}
       \label{eq:Z-tilde-Z-2}
       \begin{split}
            \Cov(Z_N^{\beta}(u,z), Z_N^{\beta}(v,w))  =  \e^{\lambda(2\beta) -2\lambda(\beta)} &\Cov(\tilde Z_N^{\beta}(u,z), \tilde Z_N^{\beta}(v,w)) \\
            & \quad + \sigma^2(\beta) \bP_u(S_N=z) \bP_v(S_N=w)\,.  
       \end{split}
\end{equation}
Since \(1+\sigma^2(\beta)=\e^{\lambda(2\beta)-2\lambda(\beta)}\) is bounded uniformly for \(\beta\in(0,\beta_0/4)\), this prefactor will be absorbed into the constants below.

One can then expand the product in \(\tilde Z_N^{\beta}(u,z)\): one can write 
\[
\tilde Z_N^{\beta}(u,z) = \sum_{k=0}^{\infty}  \sum_{0 < n_1< \cdots < n_{k}< N } \sum_{z_1,\ldots, z_{k}\in \Z^2} \bP_u \big( S_{n_i}=z_i \text{ for all } 1\leq i\leq k+1 \big)\prod_{i=1}^k \xi_{n_i,z_i} \,,
\]
with the convention \((n_{k+1},z_{k+1}) = (N,z)\), and \(\prod_{i=1}^k \xi_{n_i,z_i} =1\) if \(k=0\).

We can in fact write this chaos expansion in a more compact way.
Let us denote by \(\mathcal{I}_{N}\) the class of subsets \(A \subset \{1,\ldots, N-1\} \times \Z^2\) with different time indices, so that \(A\) can be written as \(A =\{(n_1,z_1), \ldots, (n_k,z_k)\}\) with \(0<n_1<\cdots <n_k < N\). 
Then, for such \(A \in \mathcal{I}_{N}\), let us denote by 
\[
q_{(0,u),(N,z)}(A) := \bP_u\big( S_{n_i}=z_i \text{ for all } 1\leq i\leq k, \, S_N=z \big)\,, \qquad \xi(A) := \prod_{(n_i,z_i)\in A} \xi_{n_i,z_i}\,.
\]
Thus, we may rewrite the chaos decomposition above as follows:
\begin{equation}
       \label{def:chaos}
       \tilde Z_N^{\beta}(u,z) = \sum_{A \in  \mathcal{I}_{N}} q_{(0,u),(N,z)}(A) \, \xi(A) \,.  
\end{equation}
This expression is very useful, in particular once one notices that the variables \((\xi(A))_{A\in \mathcal{I}_{N}}\) are orthogonal in \(L^2(\PP)\), allowing us to use~\eqref{def:chaos} to compute covariances.

In particular, since \(\EE[\xi(A)^2] = \sigma^2(\beta)^{|A|}\) by definition of \(\sigma^2(\beta)\), one gets that 
\begin{equation}
       \label{formula-covariance}
       \Cov(\tilde Z_N^{\beta}(u,z), \tilde Z_N^{\beta}(v,w)) 
       = \sum_{A \in  \mathcal{I}_{N}, A\neq \emptyset} \sigma^2(\beta)^{|A|} \, q_{(0,u),(N,z)}(A)  q_{(0,v),(N,w)}(A) \, .
\end{equation}
Note that a similar decomposition as~\eqref{def:chaos} also holds for the constrained partition function: for \(u\in \Lambda_N^x\), one simply needs to replace \(q_{(0,u),(N,z)}(A)\) by
\[
q_{(0,u),(N,z)}^{\mathrm{cons}}(A) = \bP_u \big( S_{n_i}=z_i \text{ for all } 1\leq i\leq k+1 , \max_{1\leq j \leq N} |S_j|_\infty < 5 \sqrt{N}\big) \,.
\]
The same covariance computation as before holds, and since \( q_{(0,u),(N,z)}^{\mathrm{cons}}(A) \leq  q_{(0,u),(N,z)}(A)\), one clearly has that
\[
\Cov(Z_N^{\beta,\mathrm{cons}}(u,z), Z_N^{\beta,\mathrm{cons}}(v,w))  \leq \Cov(Z_N^{\beta}(u,z), Z_N^{\beta}(v,w)) \,,
\]
which is the first part of Lemma~\ref{lem:covariance}.
We can therefore focus on the unconstrained case.

\subsection{Proof of~\texorpdfstring{\eqref{ineq:covariances}}{}}

Let us first argue that we may focus on \(\Cov(\tilde Z_N^{\beta}(u,z), \tilde Z_N^{\beta}(v,w))\).
In view of~\eqref{eq:Z-tilde-Z}-\eqref{eq:Z-tilde-Z-2}, we only need to bound the second term in~\eqref{eq:Z-tilde-Z-2}.

Let \(p_n(x) = \bP(S_n=x)\) be the random walk transition kernel, and let us state the following uniform bound, valid for the simple random walk in dimension \(d=2\): there exists a constant \(C\) such that for any \(n\geq 1\) and \(x\in \Z^2\),
\begin{equation}
       \label{eq:bound-SRW}
       p_n(x) \leq \frac{C}{n} \,\e^{-\frac{|x|^2}{n}} \,.
\end{equation}
This follows, for instance, from the corresponding one-dimensional estimate after rotating the lattice coordinates --- for the one-dimensional simple random walk, it can be deduced from Stirling's formula, after having expressed the probability with binomial coefficients.
With this bound, we get that for \(w=z\),
\[
       \bP_u(S_N=z) \bP_v(S_N=w) \leq  \frac{C^2}{N^2}  \e^{ - \frac{|z-u|^2 + |z-v|^2}{N}} \leq   \frac{C^2}{N^2}  \e^{ - \frac{|v-u|^2}{2N}} \,,
\]
since \(|z-u|^2+|z-v|^2 \geq \frac12 |v-u|^2\) for all \(z\in \R^2\) (the minimum is reached at \(z= \frac{u+v}{2}\)).
This shows that the second term in~\eqref{eq:Z-tilde-Z-2} is bounded by a constant times \(\frac{\sigma^2(\beta)}{N^2} K_N(v-u) K_N(0)\), recalling the definition~\eqref{def:K-N} of \(K_N(x)\). { For $\beta$ sufficiently small so that $\sigma^2(\beta)<\mathcal V(e^{\vartheta_0})$, we conclude that this contribution is bounded by a constant times $\frac{\mathcal V(e^{\vartheta_0})}{N^2}K_N(v-u)K_N(w-z)$, since $w=z$.}

\begin{remark}
       \label{rem:compare-K-Q}
Let us note that, thanks to~\eqref{eq:bound-SRW}, we have that for any \(N\in \N\) and \(x\in \Z^2\)
\begin{equation*}
       Q_N(x) := \sum_{n=1}^{N} p_{2n}(x) \leq  2C  \sum_{n=1}^N n^{-1} \e^{- \frac{|x|^2}{2n}} \leq C'\Big(1+\log_+\Big(\frac{N}{1+|x|^2} \Big) \Big)\, \e^{- \frac{|x|^2}{2N}} =  C' K_N(x) \,,
\end{equation*}
where we have used that the main contribution of the sum comes from \(n\) between \(|x|^2\) and \(N\) whenever \(|x|^2\leq N\). 
\end{remark}

Let us now focus on bounding \(\Cov(\tilde Z_N^{\beta}(u,z), \tilde Z_N^{\beta}(v,w))\).
We start from the formula~\eqref{formula-covariance}.
If \(A \in \mathcal{I}_N\) is non-empty and \(A =\{(n_1,z_1), \ldots, (n_k,z_k)\}\) with \(0<n_1<\cdots <n_k < N\), let us denote by \(\mathrm{first}(A) = (n_1,z_1)\) and \(\mathrm{last}(A) = (n_k,z_k)\).
Note also that, using the Markov property,
\[
q_{(0,u),(N,z)}(A) = p_{n_1}(z_1-u) \hat{q}(A) p_{N-n_k}(z-z_k) \quad \text{ with } \quad \hat q(A) = \prod_{i=2}^{k} p_{n_i-n_{i-1}}(z_i-z_{i-1}) \,.
\]
Therefore, decomposing~\eqref{formula-covariance} over the first and the last points of~\(A\), we obtain that the covariances \(\Cov(\tilde Z_N^{\beta}(u,z), \tilde Z_N^{\beta}(v,w))\) are equal to
\begin{equation}
       \label{eq:covariance-0}
       \sum_{y,y' \in \Z^2} \sum_{1\leq n \leq m \leq N-1} 
       p_n(y-u) p_{n}(y-v) U_{n,m}^{\beta} (y,y') p_{N-m}(z-y') p_{N-m}(w-y') \,,     
\end{equation}
where we have introduced
\[
U_{n,m}^{\beta} (y,y'):= 
\sum_{ \mathrm{first}(A) = (n,y) ,\, \mathrm{last}(A) = (m,y') } \sigma^2(\beta)^{|A|} \, \hat q(A)^2\,.
\]
Here the case \(m=n\) is understood as the singleton case, so that \(U_{n,n}^{\beta}(y,y')=\sigma^2(\beta)\ind_{\{y=y'\}}\); we use this convention in what follows.

Let us also note that for \(n<m\) we can rewrite \(U_{n,m}^{\beta} (y,y')\) as follows: by translation invariance, with the convention \(n_0=0,z_0=0\), 
\[
\begin{split}
       U_{n,m}^{\beta} (y,y')
       &  =  \sigma^2(\beta) \sum_{k=1}^{\infty}\sigma^2(\beta)^k \sum_{1\leq n_1<\cdots < n_k = m-n} \sum_{\substack{z_1,\ldots, z_k \in \Z^2\\ z_k=y'-y}} \prod_{i=1}^k p_{n_i-n_{i-1}}(z_i-z_{i-1})^2 \\
       & =\sigma^2(\beta) \sum_{k=1}^{\infty} (\sigma^2(\beta) R_N)^k \sum_{1\leq n_1<\cdots < n_k = m-n} \sum_{\substack{z_1,\ldots, z_k \in \Z^2\\ z_k=y'-y}} \prod_{i=1}^k \frac{p_{n_i-n_{i-1}}(z_i-z_{i-1})^2}{R_N} \,,
\end{split}
\]
where we recall that \(R_N= \sum_{n=1}^N p_{2n}(0) = \sum_{n=1}^{N} \sum_{x\in \Z^2} p_n(x)^2\).
In particular, we obtain that
\begin{equation}
       \label{eq:renewal-density}
       U_{n,m}^{\beta} (y,y') = \sigma^2(\beta) \mathbf{U}_{N,\lambda} (m-n,y'-y) \,,
\end{equation}
where \(\lambda = \sigma^2(\beta) R_N\) and \(\mathbf{U}_{N,\lambda} (n,z)\) is an exponentially weighted renewal density, introduced in~\cite[Eq.~(3.21)]{CSZ19b}. (This formula also holds when \(n=m\), in which case \(\mathbf{U}_{N,\lambda} (0,y'-y)= \ind_{\{y=y'\}}\).)

We now split the sum~\eqref{eq:covariance-0} into three parts:
\[
 \text{\eqref{eq:covariance-0}} \leq \mathbf{I} + \mathbf{II} + \mathbf{III} \,,
\]
where the term \(\mathbf{I}\) corresponds to a sum restricted to \(n\geq N/3\), \(\mathbf{II}\) to a sum restricted to \(N-m\geq N/3\), and \(\mathbf{III}\) to a sum restricted to \(m-n\geq N/3\).


\smallskip
\textit{Term \(\mathbf{III}\). }
For this term, let us simply state Theorem~2.3 from~\cite{CSZ19b} (see in particular Equation~(2.3)), which directly gives that there is a universal constant \(C\) such that, for any \(1\leq j \leq N\) and any \(y\in \Z^2\), if \(N\) verifies~\eqref{def:critical-window} and sufficiently large,
\[
\sigma^2(\beta) \mathbf{U}_{N,\lambda} (j,y) \leq \frac{C}{N} \, \frac{1}{j} G_{\vartheta_0+\frac12}\Big(\frac{j}{N}\Big) \,,
\]
where \(G_{\vartheta_0+\frac12}(t)\) is the weighted renewal density of the Dickman subordinator.
For \(t\in (0,1]\), we have the formula \(G_{\vartheta_0+\frac12}(t) = \int_0^{\infty} \frac{\e^{(\vartheta_0+\frac12-\gamma)s} s t^{s-1}}{\Gamma(s+1)} \dd s\), with \(\gamma\) the Euler--Mascheroni constant.
In particular, in the term \(\mathbf{III}\), the sum is bounded by 
\[
\begin{split}
\mathbf{III} \leq \frac{C}{N^2} \sup_{t \in [\frac13,1]} G_{\vartheta_0+\frac12}(t) \sum_{y,y'\in \Z^2}  \sum_{1\leq n\leq m\leq N-1} & p_n(y-u)p_n(y-v) p_{N-m}(z-y') p_{N-m}(w-y')\\
& \leq \frac{C' \mathcal{V}(\e^{\vartheta_0})}{N^2}  \sum_{1\leq n \leq m \leq N-1} p_{2n}(v-u) p_{2(N-m)}(w-z) \,,
\end{split}
\]
where we have used the Chapman--Kolmogorov property (and the symmetry of the random walk) to get that \(\sum_{y\in \Z^2} p_n(y-u)p_n(y-v)  = p_{2n}(v-u)\) and similarly for the sum over~\(y'\).
We have also used that \(G_{\vartheta_0+\frac12}(t) \leq C\int_0^{\infty} \frac{\e^{\vartheta_0 s}}{\Gamma(s+1)}\dd s = C\mathcal{V}(\e^{\vartheta_0})\), since \(e^{(\frac12-\gamma) s}s t^{s-1}\) is bounded uniformly for \( t \in [\frac13,1]\) and \(s>0\) (note that \(\frac12-\gamma<0\)).
Enlarging the sums, we end up with\
\[
\mathbf{III} \leq \frac{C' \mathcal{V}(\e^{\vartheta_0})}{N^2}  \Big( \sum_{n=1}^{N} p_{2n}(v-u) \Big) \Big( \sum_{n=1}^{N} p_{2n}(w-z) \Big) \,,
\]
which gives the result thanks to Remark~\ref{rem:compare-K-Q}.

\smallskip
\textit{Terms \(\mathbf{I}\) and \(\mathbf{II}\). }
The two terms are symmetric, so let us treat only the term \(\mathbf{I}\). 
We may again use~\eqref{eq:bound-SRW}, to get that 
we get that
\[
p_n(y-u)p_n(y-v) \leq  \frac{C^2}{n^2}\, \e^{- \frac{|y-u|^2+|y-v|^2}{n}} \leq  \frac{C^2}{n^2}\, \e^{- \frac{|v-u|^2}{2 n}} \,.
\]
This gives that
\[
\begin{split}
   \mathbf{I} \leq  \sum_{n=N/3 }^{N-1} \sum_{m=n}^{N-1} \frac{C^2}{n^2}\, \e^{- \frac{|v-u|^2}{2 n}} & \sum_{y,y'\in \Z^2}  \sigma^2(\beta) \mathbf{U}_{N,\lambda} (m-n,y'-y) p_{N-m}(z-y')p_{N-m}(w-y') \\  
   & \quad \leq \frac{C'}{N^2} \e^{- \frac{|v-u|^2}{2 N}} \sum_{n=N/3 }^{N-1} \sum_{m=n}^{N-1} \sigma^2(\beta)\, U_{N,\lambda}(m-n) p_{2(N-m)}(z-w) \,,
\end{split}
\]
where we have summed over \(y\), then over \(y'\), with \(U_{N,\lambda}(m-n) = \sum_{z\in \Z^2} \mathbf{U}_{N,\lambda} (m-n,z)\).
Now, enlarging the sum, we get that
\[
       \mathbf{I} \leq \frac{C'}{N^2} \e^{- \frac{|v-u|^2}{2 N}} \Big( \sigma^2(\beta)\sum_{n=0}^N U_{N,\lambda}(n)\Big) \Big(  \sum_{n=1}^{N} p_{2n}(z-w)\Big)\,.
\]
Then, notice that \cite[Theorem~1.4]{CSZ19b} (in particular (1.19)) gives that, for \(N\) verifying~\eqref{def:critical-window} and sufficiently large,
\[
\sigma^2(\beta)\sum_{n=0}^N U_{N,\lambda}(n) \leq \sigma^2(\beta)  + \frac{C}{N} \sum_{n=1}^N G_{\vartheta_0+\frac12}\Big(\frac{n}{N}\Big) \leq  C' \int_0^1 G_{\vartheta_0+\frac12}(t)\dd t \,,
\]
with \(\int_0^1 G_{\vartheta_0+\frac12}(t)\dd t = \int_{0}^{\infty} \frac{\e^{(\vartheta_0+\frac12-\gamma)s}}{\Gamma(s+1)} \dd s = \mathcal{V}(\e^{\vartheta_0+\frac12-\gamma})\).
Since, \(\frac12-\gamma\leq 0\), this gives that
\[
 \mathbf{I} \leq \frac{C'\mathcal{V}(\e^{\vartheta_0}) }{N^2} \e^{- \frac{|v-u|^2}{2 N}} \Big( \sum_{n=1}^{N} p_{2n}(z-w)\Big) \leq   \frac{C''\mathcal{V}(\e^{\vartheta_0}) }{N^2}  K_N(v-u)K_N(z-w)\,,
\]
using again Remark~\ref{rem:compare-K-Q} for the last inequality. This concludes the proof.
\qed

\section{Proof of Theorems~\ref{thm:second-moment} and \ref{thm:Lyapunov}}

\label{sec:Lyapunov}

Let us first quickly prove the existence of Lyapunov exponents of order \(p\).
We only treat the case \(p\geq 1\), the case \(p<1\) being similar.
The proof simply relies on the sub-multiplicativity of \(\EE[(Z_{N}^{\beta})^p]\), using the Markov property of the partition function, see \cite[Eq.~(2.3)]{Com17}:
for \(N,M\in \N\), it can be written as follows
\[
       Z_{N+M}^{\beta,\omega} =  Z_{N}^{\beta,\omega} \times \bE_N^{\beta,\omega}\big[ Z_{(N,N+M]}^{\beta,\omega}(S_N,\Z^2) \big]\,,
\]
where we recall the notation \(Z_{(N,N+M]}^{\beta,\omega}(x,\Z^2)\) for the point-to-plane partition function with starting point \((N,x)\), between times \(N\) and \(N+M\). 
Since \(p\geq 1\), we obtain by Jensen's inequality that 
\(
       (Z_{N+M}^{\beta,\omega})^p \leq  (Z_{N}^{\beta,\omega})^p \times \bE_N^{\beta,\omega}\big[ Z_{(N,N+M]}^{\beta,\omega}(S_N,\Z^2)^p \big] .
\)
Then, taking the conditional expectation with respect to \(\mathcal{F}_N\), and using that \(Z_{(N,N+M]}^{\beta,\omega}(S_N,\Z^2)\) is independent of \(\mathcal{F}_N\) and has the same distribution as \(Z_{M}^{\beta,\omega}\) by translation invariance, we get that 
\[
\EE\big[(Z_{N+M}^{\beta,\omega})^p \mid \mathcal{F}_N \big] \leq  (Z_{N}^{\beta,\omega})^p \times \EE\big[ (Z_{M}^{\beta,\omega})^p \big] \,.
\]
Taking the expectation leads to the sub-multiplicativity of \(\EE[(Z_{N}^{\beta,\omega})^p]\), or equivalently the sub-additivity of its logarithm:
\(
       \log \EE\big[(Z_{N+M}^{\beta,\omega})^p \big] \leq \log \EE\big[(Z_{N}^{\beta,\omega})^p \big] + \log \EE\big[(Z_{M}^{\beta,\omega})^p \big]  .
\)
By Fekete's lemma, we conclude that, for any \(p\geq 1\) and $\beta \in (0,\beta_0/p)$,
\begin{equation}
       \label{eq:lyapunov-inf}
       \tF_p(\beta) = \lim_{N\to\infty} \frac1N \log \EE[(Z_{N}^{\beta,\omega})^p] = \inf_{N\in \N} \frac1N \log \EE[(Z_{N}^{\beta,\omega})^p]\,.
\end{equation}
A similar statement holds for \(p\in (0,1)\) with the infimum replaced by a supremum, since \((\EE[(Z_{N}^{\beta,\omega})^p])_{N\geq 1}\) is then super-multiplicative, because Jensen's inequality is reversed.

\subsection{Proof of Theorem~\ref{thm:second-moment}}

The proof follows from an explicit formula for the Lyapunov exponent of the second moment. 
Recall that, using the replica trick, we have that 
\[
\EE\big[ (Z_N^{\beta,\omega})^2\big] = \bE \Big[ \e^{ (\lambda(2\beta)-2\lambda(\beta)) \sum_{n=1}^N \ind_{\{S_{2n}=0\}}}\Big] \,,
\]
which is the partition function of a homogeneous pinning model.
Thus, the free energy can be implicitly expressed as follows, see \cite[\S2.1]{Gia07} for a reference.
Denoting \(\tau_1 := \min\{n\geq 1: S_{2n}=0\}\), for any \(\beta\in (0,\beta_0/2)\), the Lyapunov exponent \(\tF_2(\beta)\) is the unique positive solution of
\[
\sum_{n=1}^{\infty} \e^{- \tF_2(\beta) n} \bP(\tau_1=n) = \e^{-(\lambda(2\beta)-2\lambda(\beta))} \,.
\]
Introduce \(\hat U(f) := \sum_{n=0}^{\infty} \e^{-f n} \bP(S_{2n}=0)\) and notice that \(\hat U(f) =1+ \hat U(f)\sum_{n=1}^{\infty} \e^{- f n} \bP(\tau_1=n)\), thanks to the renewal equation \(\bP(S_{2n}=0) = \ind_{\{n=0\}} + \sum_{k=1}^n \bP(\tau_1 =k) \bP(S_{2(n-k)}=0)\).
Thus, the above is equivalent to 
\begin{equation}
       \label{eq:formula-F2}
       \hat U(\tF_2(\beta)) = \sum_{n=0}^{\infty} \e^{- \tF_2(\beta) n} \bP(S_{2n}=0) = \frac{1}{1-\e^{-(\lambda(2\beta)-2\lambda(\beta))}} = \frac{1+\sigma^2(\beta)}{ \sigma^2(\beta)} \,,
\end{equation}
recalling also that \(\sigma^2(\beta) =\e^{\lambda(2\beta)-2\lambda(\beta)}-1\).

Now, since \(\bP(S_{2n}=0) = 16^{-n} \binom{2n}{n}^2\), it turns out that \(\hat U(f)\) is (a multiple of) the complete elliptic function of the first kind (which is a special case of the hypergeometric function).
More precisely, \cite[Eq.~\href{https://dlmf.nist.gov/19.2.E8}{(19.2.8)} and Eq.~\href{https://dlmf.nist.gov/19.5.E1}{(19.5.1)}]{DLMF} give the identity
\[
\hat U(f) = \frac{2}{\pi} \int_0^{\pi/2} \frac{\dd t}{ \sqrt{1- \e^{-f} \sin^2(t)}} \,.
\]
Notice that the behavior of the integral near its singularity is known, see~\cite[Eq.~\href{https://dlmf.nist.gov/19.12.E1}{(19.12.1)}]{DLMF}: one gets that as \(f\downarrow 0\)
\[
\begin{split}
\pi \hat U(f) & = \log\Big(\frac{16}{1-\e^{-f}}\Big) + \frac14 (1-\e^{-f}) \Big( \log\Big(\frac{16}{1-\e^{-f}}\Big) -2\Big) + O\Big( (1-\e^{-f})^2 \log\Big(\frac{16}{1-\e^{-f}}\Big) \Big) \\
& = \log\Big(\frac{16}{f}\Big) + \frac{f}{4} \log\Big(\frac{16}{f}\Big) + O\Big( f^2 \log \frac{1}{f} \Big) \,.
\end{split}
\]
Using this expression with \(f = \tF_2(\beta)\) in~\eqref{eq:formula-F2} and recalling that we introduced \(\alpha(\beta) := \pi \frac{1+\sigma^2(\beta)}{\sigma^2(\beta)}\), since \(\tF_2(\beta)\downarrow 0\) as \(\beta\downarrow 0\), we get
\[
\alpha(\beta) = \pi \hat U(\tF_2(\beta)) = \log\Big(\frac{16}{\tF_2(\beta)}\Big) + \frac{\tF_2(\beta)}{4} \log\Big(\frac{16}{\tF_2(\beta)}\Big) + O\Big( \tF_2(\beta)^2 \log \frac{1}{\tF_2(\beta)} \Big)\,.
\]
Inverting this relation, we end up with 
\[
\frac{1}{\tF_2(\beta)} = \frac{1}{16}\, \e^{\alpha(\beta)} - \frac{1}{4}\, \alpha(\beta) +O\big(\alpha(\beta)^2\, \e^{-\alpha(\beta)}\big) \qquad \text{ as } \beta\downarrow 0 \,,
\]
which is exactly what is claimed in Theorem~\ref{thm:second-moment}.
From this, we then easily deduce that \(\tF_2(\beta) \sim 16 \, \e^{-\alpha(\beta)}  = 16\, \e^{-\pi}\, \e^{-\frac{\pi}{\sigma^2(\beta)}}\) as \(\beta\downarrow 0\).
\qed

\begin{remark}
       Let us stress that one could push the expansion of \(\e^{-\alpha(\beta)}/\tF_2(\beta)\) to any arbitrary power of \(\alpha(\beta) \e^{-\alpha(\beta)}\).
       We have however decided to present this expansion only up to the second order, for simplicity of exposition.
\end{remark}

\subsection{Proof of Theorem~\ref{thm:Lyapunov}}
\label{sec:moments}

Before we start the proof, let us stress that \(p\mapsto \tF_p(\beta)\) is a convex function of \(p\), as a limit of convex functions.
In fact, all bounds will follow from convexity arguments, combined with Theorems~\ref{thm:fe} and~\ref{thm:second-moment} or existing results.

\smallskip
\noindent
\textit{Item \ref{i-p-01}, \(p\in (0,1)\).}
For the upper bound, we use that \(\EE[Z^p] \leq 2^p \EE[ Z\wedge 1]^{p\wedge (1-p)}\) for any non-negative random variable \(Z\) with \(\EE[Z]=1\), see \cite[Lem.~3.3]{BCT25}.
We thus have that 
\[
\tF_p(\beta) = \lim_{N\to\infty} \frac1N \log \EE[(Z_{N}^{\beta,\omega})^p]  \leq p\wedge(1-p) \limsup_{N\to\infty} \frac1N \log \EE[Z_{N}^{\beta,\omega}\wedge 1]\,.
\]
Then, by \cite[Theorem~2.2]{BCT25}, there is a constant \(c_2>0\) such that, for all \(\beta\in (0,\beta_0/4)\), 
\[
\sup_{N\geq 1} \frac1N \log \EE[Z_{N}^{\beta,\omega}\wedge 1] \leq - c_2\, \e^{- \frac{\pi}{\sigma^2(\beta)}} \,.
\]
This concludes the upper bound in~\ref{i-p-01}.

For the lower bound, we use two inequalities. 
First, we have \(\EE[Z^p] \geq \e^{p \EE[\log Z]}\) by Jensen's inequality, so that 
\[
\tF_p(\beta) = \lim_{N\to\infty} \frac1N \log \EE[(Z_{N}^{\beta,\omega})^p]  \geq p \lim_{N\to\infty} \frac1N \EE\big[\log Z_{N}^{\beta,\omega}\big] = p\, \tf(\beta)  \geq - c\, p\, \e^{- \frac{\pi}{\sigma^2(\beta)}}\,,
\]
the last inequality following from Theorem~\ref{thm:fe}.
Second, by Hölder's inequality, for any \(p\in (0,1)\) we have that \(1= \EE[Z^{p^2} Z^{1-p^2}] \leq \EE[Z^p]^{p} \EE[Z^{1+p}]^{1-p}\).
Thus, we get that 
\[
\tF_p(\beta) \geq  - \frac{1-p}{p}\, \tF_{1+p}(\beta) \geq - \frac{1+p}{2p} (1-p) \tF_{2}(\beta) \geq  - \frac{c (1+p)}{2p} (1-p)  \, \e^{- \frac{\pi}{\sigma^2(\beta)}}\,,
\]
where for the second inequality we have used that \(\EE[Z^{1+p}] \leq \EE[Z^2]^{\frac{1+p}{2}}\), and then applied Theorem~\ref{thm:second-moment} for the last one.
Combining the two inequalities (the first one for \(p\) close to \(0\), the second one for \(p\) close to \(1\)), we get the lower bound in~\ref{i-p-01}.

\smallskip
\noindent
\textit{Item \ref{ii-p>1}, \(p>1\).}
By convexity of \(p\mapsto \tF_p(\beta)\), we have that 
\[
 -2(p-1) \tF_{1/2}(\beta) \leq \tF_p(\beta) \leq (p-1) \tF_{\lceil p\rceil}(\beta) \,.
\]
The left-hand side uses that the slope at \(p=1\) is lower bounded by \(2(\tF_1(\beta) -\tF_{1/2}(\beta)) = -2 \tF_{1/2}(\beta)\). 
       The right-hand side uses that \(\tF_p(\beta)\) lies below the chord between \(\tF_1(\beta)\) and \(\tF_{\lceil p \rceil}(\beta)\).

We can thus directly use item~\ref{i-p-01} for the lower bound.
The upper bound is  \((p-1) \tF_{\lceil p\rceil}(\beta)\) if \(p\in (1,2]\) and we can thus apply Theorem~\ref{thm:second-moment}.
\qed

\subsection*{Acknowledgements}
This work has been initiated during our stay at IMS in Singapore during the program ``Statistical Mechanics and Singular SPDEs'' (4--20 May 2026).
We are very grateful to the organizers and to the IMS for their support.
We would also like to thank in particular Francesco Caravenna and Nicola Turchi for several insightful discussions on this topic (in particular regarding, but not limited to, Theorem~\ref{thm:second-moment}), Sébastien Martineau for comments on percolation, as well as Stefan Junk, Hubert Lacoin, Rongfeng Sun, and Nikos Zygouras.

Q.B.\ acknowledges the support of Institut Universitaire de France and ANR Local (ANR-22-CE40-0012-02).
S.N.\ acknowledges support from JSPS KAKENHI Grant Numbers JP24K16937 and JP25K00911.

\bibliographystyle{abbrv}
\bibliography{biblio}

\end{document}